\documentclass[a4paper,12pt]{amsart}
\usepackage{amssymb}
\usepackage{mathrsfs}
\usepackage{amsmath,amssymb,amsthm,latexsym,amscd,mathrsfs}
\usepackage{indentfirst}
\usepackage{stmaryrd}
\usepackage{graphicx}
\usepackage{subfigure}
\usepackage{extarrows}
\usepackage{amssymb}
\usepackage{amsmath,amssymb,amsthm,latexsym,amscd,mathrsfs}
\usepackage{indentfirst}
\usepackage{multirow}
\usepackage{latexsym}
\usepackage{amsfonts}
\usepackage{color}
\usepackage{pictexwd,dcpic}
\usepackage{graphicx}
\usepackage{psfrag}
\usepackage{hyperref}
\usepackage{comment}
\usepackage{cases}

\makeatletter
\@namedef{subjclassname@2010}{{\rm2020} Mathematics Subject Classification}
\makeatother

 \newcommand{\ROM}[1]{\mathrm{\uppercase\expandafter{\romannumeral#1}}}
  \theoremstyle{definition}
   \numberwithin{equation}{section} \theoremstyle{plain}

 \newtheorem{thm}{Theorem}[section]
 \newtheorem{lem}{Lemma}[section]
 
 \newtheorem{cor}{Corollary}[section]
  
 \newtheorem{rem}{Remark}[section]
 \newtheorem{prop}{Proposition}[section]

  \makeatletter

  \numberwithin{equation}{section}
  \allowdisplaybreaks
\makeatother
\title[On the Chern conjecture for isoparametric hypersurfaces]{\textbf{On the Chern conjecture for isoparametric hypersurfaces}}
\author[Z. Z. Tang]{Zizhou Tang}\address{Chern Institute of Mathematics, Nankai University, Tianjin 300071, P. R. China}
\email{zztang@nankai.edu.cn}
\author[W. J. Yan]{Wenjiao Yan$^{\dag}$}\address{School of Mathematical Sciences, Laboratory of Mathematics and Complex Systems, Beijing Normal University, Beijing, 100875, P. R. China}
\email{wjyan@bnu.edu.cn}
\thanks {$^{\dag}$ the corresponding author}
\thanks {The project is partially supported by the NSFC (No.11722101, 11871282, 11931007), BNSF (Z190003), and Nankai Zhide Foundation.}


\begin{document}

\maketitle

\begin{abstract} For a closed hypersurface $M^n\subset S^{n+1}(1)$ with constant mean curvature and constant non-negative scalar curvature, the present paper shows that if $\mathrm{tr}(\mathcal{A}^k)$ are constants for $k=3,\ldots, n-1$ for shape operator $\mathcal{A}$, then $M$ is isoparametric. The result generalizes the theorem of de Almeida and Brito \cite{dB90} for $n=3$ to any dimension $n$, strongly supporting Chern's conjecture.
\end{abstract}

\section{\textbf{Introduction}}
 After more than 50 years of extensive research, the famous Chern conjecture for isoparametric hypersurfaces in spheres is still an unsolved challenging problem. S. T. Yau raised it again as the 105th problem in his Problem Section \cite{Yau82}. Mathematicians are constantly engaged in this problem. See the excellent survey on this topic by M. Scherfner, S. Weiss and  S. T. Yau \cite{SWY12} and \cite{GT12} by J. Q. Ge and the first author.

\vspace{2mm}
\noindent
\textbf{Chern's conjecture.} \,\emph{Let $M^n$ be a closed, minimally immersed hypersurface of the unit sphere $S^{n+1}(1)$ with constant scalar curvature. Then $M^n$ is isoparametric.}
% $R_M$ (or equivalently, constant $S$--the squared norm of its second fundamental form).
%Then for each $n$, the set of all possible values for $R_M$ is discrete.
\vspace{2mm}

It was originally proposed in a less strong version by S. S. Chern in \cite{Che68} and \cite{CdK70}. The original version of this conjecture relates to the remarkable theorem of J. Simons \cite{Sim68}:
% \cite{Sim68}: for a closed minimal immersed hypersurface in the unit sphere, denoting by $S$ the squared norm of its second fundamental form, if $0\leq S\leq n$, then either $S\equiv 0$ or $S\equiv n$ on $M^n$. 
\vspace{2mm}

\noindent
\textbf{Simon's theorem}  \emph{\,\,Let $M^n\subset S^{n+1}(1)$ be a closed, minimally immersed hypersurface and $S$ the squared norm of its second fundamental form. Then $$\int_M(S-n)S\geq 0.$$
In particular, for $S\leq n$, one has either $S\equiv 0$ or $S\equiv n$ on $M^n$. }
\vspace{2mm}

Notice that for a closed hypersurface $M^n$ in the unit sphere with constant mean curvature, $S$ is constant if and only if the scalar curvature $R_M$ is constant. In the minimal case, it follows from Simon's theorem that $S=0$ or $S\geq n$, which led S. S. Chern to propose the following original conjecture:
\vspace{2mm}

\noindent
\textbf{Conjecture.} \,\,\emph{Let $M^n\subset S^{n+1}(1)$ be a closed, minimally immersed hypersurface with constant scalar curvature $R_M$. Then for each $n$, the set of all possible values for $R_M$ (or equivalently $S$) is discrete.} 
\vspace{2mm}

Actually, the minimal hypersurfaces with constant $S$ in Simon's theorem can be characterized clearly: those with $S\equiv 0$ are the equatorial $n$-spheres in $S^{n+1}$, and those with $S\equiv n$ are characterized by \cite{CdK70} and \cite{Law69} independently that $M^n$
must be the Clifford tori $S^k(\sqrt{\frac{k}{n}})\times S^{n-k}(\sqrt{\frac{n-k}{n}})$ $(1\leq k\leq n-1)$.
%S.S.Chern, M. do Carmo and S. Kobayashi \cite{CdK70} and H. B. Lawson \cite{Law69} proved independently that the Clifford tori $S^r(\sqrt{\frac{r}{n}})\times S^{n-r}(\sqrt{\frac{n-r}{n}})$ $(0<r<n)$ are the only minimal hypersurfaces with $S\equiv n$. 
In other words, they finished the first pinching problem for $S$ of closed minimal hypersurfaces in $S^{n+1}$.%which is the minimal umbilic isoparametric hypersurface; when $S\equiv n$ in Simon's theorem, it was characterized by \cite{CdK70} and \cite{Law69} independently that $M^n$ must be the Clifford tori $S^r(\sqrt{\frac{r}{n}})\times S^{n-r}(\sqrt{\frac{n-r}{n}})$ $(0<r<n)$, which are exactly isoparametric hypersurfaces in $S^{n+1}$ with two distinct principal curvatures.

In 1983, Peng and Terng \cite{PT83} \cite{PT83'} initiated the study of the second pinching problem and made the first breakthrough towards this conjecture:
\vspace{2mm}

\noindent
\textbf{Peng-Terng's theorem} \,\, \emph{Let $M^n$ $(n\geq 3)$ be a closed, minimally immersed hypersurface in $S^{n+1}$ with $S=constant$. If $S>n$, then $S>n+\frac{1}{12n}$. In particular, when $n=3$, if $S>3$, then $S\geq 6$.}
\vspace{2mm}

Peng and Terng has already obtained the optimal result in the case $n=3$, because the equality $S=6$ is achieved by certain minimal isoparametric hypersurfaces $M^3\subset S^4$.
During the past three decades, 
%there have been some important progresses on this second pinching problem. For example, 
Yang-Cheng \cite{YC98}  and Suh-Yang \cite{SY07} improved the second pinching constant from $\frac{1}{12n}$ to $\frac{3n}{7}$. However, it is still an open problem for higher dimensional case that  if $S>n$ and $S$ is constant, then $S\geq 2n$? Without assuming $S=constant$, there are also results on this second pinching problem, for more details, please see \cite{DX11}.

As a matter of fact, up to now, the only known examples for minimal hypersurfaces with constant $S$ in spheres are isoparametric hypersurfaces.  Based on this, Verstraelen, Montiel, Ros and Urbano \cite{Ver86} firstly formulated the stronger version of the Chern conjecture given at the beginning of this paper. For a more general version of the Chern conjecture, see for example \cite{LXX17}.

From another aspect, the Chern conjecture is also closely related with another famous conjecture of S.T. Yau on the first eigenvalue. Tang-Yan \cite{TY13} proved Yau's conjecture in the isoparametric case, that is, for a closed minimal isoparametric hypersurface $M^n$ in $S^{n+1}(1)$, the first eigenvalue of the Laplace operator is equal to the dimension $n$. Consequently, if Chern's conjecture is proved, Yau's conjecture for the minimal hypersurface with constant scalar curvature would also be right.

Now let us briefly review a few facts about isoparametric theory. The isoparametric hypersurfaces in spheres are defined to be hypersurfaces whose principal curvature functions are constant. The classification of them is listed as the 34th problem in S. T. Yau's ``Open Problems in Geometry" \cite{Yau14} and completed recently. Due to the celebrated result of M\"{u}nzner \cite{Mun80}, if $M^n$ is a compact minimal isoparametric hypersurface in $S^{n+1}$, the number $g$ of pairwise distinct principal curvatures can be only $1, 2, 3, 4$ and $6$, and $S=(g-1)n$ (which is pointed out by Peng-Teng \cite{PT83}). The minimal isoparametric hypersurfaces with $g=1$ and $2$ are those with $S\equiv 0$ and $n$ mentioned in Simon's theorem. The isoparametric hypersurfaces with $g=3$ are classified by E. Cartan.
%tubes of constant radius around the minimal Veronese embedding of $\mathbb{F}P^2$ ($\mathbb{F}=\mathbb{R}, \mathbb{C}, \mathbb{H}$ and $\mathbb{O}$) into $S^{3m+1}$ ($m=1, 2, 4$ and $8$). 
%Moreover, E. Cartan \cite{Car40} also constructed an isoparametric hypersurface $M^4$ in $S^5$ with four distinct principal curvatures. 
When $g=4$, %the possible multiplicities of principal curvatures are determined by %Abresch \cite{Abr83}, Tang \cite{Tan91}, Fang \cite{Fan99}, and Stolz \cite{Sto99}, furthermore, 
Cecil-Chi-Jensen \cite{CCJ07}, Immervoll \cite{Imm08} and Chi \cite{Chi11, Chi13, Chi20} conquered the classification in this case. 
When $g=6$, Dorfmeister-Neher \cite{DN85} and Miyaoka \cite{Miy13, Miy16} conquered the classification. For more details of the isoparametric theory, please see \cite{CR15}.

 %In recent years, the isoparametric theory in space forms has been generalized to that in Riemannian manifolds (\cite{GT13}, \cite{QT15}).
%There are also some applications of isoparametric theory, see for example, , \cite{TY15} and \cite{QTY19}.

The lowest dimension for which the Chern conjecture is non-trivial is $n=3$. In 1993, Chang \cite{Cha93} finished the proof in this case. Actually, a more general theorem has been proven:

\vspace{2mm}
\noindent
\textbf{Theorem 1 (de Almeida, Brito \cite{dB90})}\,\, \emph{Let $M^3\subset S^4$ be a closed hypersurface with constant mean curvature $H$ and constant non-negative scalar curvature $R_M$. Then $M^3$ is isoparametric.}
\vspace{2mm}

Later, Chang \cite{Cha93'} and Cheng-Wan \cite{CW93} independently generalized this result by showing that $R_M$ is always non-negative under the assumption of the theorem.

 The method of \cite{dB90} was taken to deal with $4$ and $6$ dimensional cases. In the case $n=4$, Lusala-Scherfner-Sousa \cite{LSS05} showed that a closed, minimal, Willmore hypersurface $M^4$ of $S^5$ with non-negative constant scalar curvature is  isoparametric. Denoting the $r$-th power sum of principal curvatures by $f_r$, the Willmore condition for minimal hypersurfaces with constant scalar curvature is equivalent to the condition that $f_3=0$. Under their assumption, the principal curvatures appear in form of $\lambda, \mu, -\lambda, -\mu$. Deng-Gu-Wei \cite{DGW17} generalized this result by dropping the non-negativity assumption of the scalar curvature.  
 
 In the case $n=6$, Scherfner-Vrancken-Weiss \cite{SVW12} showed that a closed hypersurface in $S^7$ with $H=f_3=f_5=0$, constant $f_4$ and constant $R_M\geq 0$ is isoparametric, which is listed as Theorem 6 in \cite{SWY12}. Under their assumption, the principal curvatures appear in form of $\lambda, \mu, \nu, -\lambda, -\mu, -\nu$. 
%One can see that their result is a special case of our Corollary \ref{TY2} which will appear below.

The authors heard that Q. M. Cheng and G. X. Wei also did some relative work in dimension $n=4$ (\cite{CW}).

The previous theorem of de Almeida and Brito is an application of another theorem of theirs with more general setting:

\vspace{2mm}

\noindent
\textbf{Theorem 2 (de Almeida and Brito \cite{dB90})}
\emph{Let $M^3$ be a closed $3$-dimensional Riemannian manifold. Suppose $\mathfrak{a}$ is a smooth symmetric $(0, 2)$ tensor field on $M^3$ and $\mathcal{A}$ is its dual tensor field of type $(1, 1)$.
Suppose in addition
\begin{itemize}
  \item [(1)] $R_M\geq 0$;
  \item [(2)] the field $\nabla \mathfrak{a}$ of type $(0, 3)$ is symmetric;
  \item [(3)] $\mathrm{tr} (\mathcal{A})$, $\mathrm{tr} (\mathcal{A}^2)$ are constants.
%  \item [(1.4)] $d(\mathrm{tr} \mathcal{A}^2)=0$
\end{itemize}
Then $\mathrm{tr} (\mathcal{A}^3)$ is a constant, and thus the eigenvalues of $\mathcal{A}$.}
\vspace{2mm}

To generalize de Almeida-Brito's Theorem 2 for dimension $3$ to any dimension, one has to conquer two difficulties:  the technical difficulty in the proof on the domain where $\mathcal{A}$ has $n$ distinct eigenvalues and the integral estimate on the domain where $\mathcal{A}$ has $g<n$ distinct eigenvalues. 

In \cite{TWY20}, Tang-Wei-Yan conquered the first difficulty
and partially generalized the results mentioned before from $n=3,4,6$ to any $n>3$:
\vspace{2mm}

\noindent
\textbf{Theorem (Tang-Wei-Yan \cite{TWY20})}
%\begin{thm}\label{thm of TWY}
\emph{Let $M^n$ $(n>3)$ be a closed $n$-dimensional Riemannian manifold on which $\int_{M}R_M\geq 0$.
Suppose that $\mathfrak{a}$ is a smooth symmetric $(0,2)$ tensor field on $M^n$, and $\mathcal{A}$ is its dual tensor field of type $(1, 1)$. If the following conditions are satisfied:
%which is Codazzian and has $n$ distinct eigenvalues $\lambda_i$ ($i=1,\cdots, n$) everywhere, %and $A$ is the $(1,1)$ tensor field corresponding to $a$ via $g$.
%in addition, $\mathrm{tr}(a^k)$ $(k=1,\cdots, n-1)$ are constants.
\begin{enumerate}
  \item[(1)] $\mathfrak{a}$ is Codazzian;
  \item[(2)] $\mathcal{A}$ has $n$ distinct eigenvalues $\lambda_1,\ldots,\lambda_n$ everywhere;
  \item[(3)] $\mathrm{tr}(\mathcal{A}^k)$ $(k=1,\ldots, n-1)$ are constants;
\end{enumerate}
then
\begin{itemize}
  \item[(a)] $\mathrm{tr}(\mathcal{A}^n)$ is a constant, i.e., $\lambda_1,\ldots, \lambda_n$ are constants;
  \item[(b)] $\int_{M}R_M\equiv 0$.
\end{itemize}
%$\mathrm{tr}(\mathcal{A}^n)$ is constant, i.e., $\lambda_i$'s $(i=1,\ldots, n)$ are constants, and $\int_{M}R_M\equiv 0$.
%\end{thm}
\vspace{2mm}
}

Taking $\mathfrak{a}$ as the second fundamental form, they immediately obtained the following:

\vspace{2mm}
\noindent
\textbf{Corollary (Tang-Wei-Yan \cite{TWY20})}
%\begin{cor}\label{cor1 of TWY}
\emph{Let $M^n$ $(n>3)$ be a closed hypersurface in the unit sphere $S^{n+1}$. If the following conditions are satisfied:
\begin{enumerate}
  \item[(1)] $R_M\geq 0$;
  \item[(2)] the principal curvatures $\lambda_1,\ldots,\lambda_n$ are distinct;
  \item[(3)] $\sum\limits_{i=1}^n\lambda_i^k$ $(k=1,\ldots, n-1)$ are constants,
\end{enumerate}
then $M^n$ is isoparametric and $R_M\equiv 0$. More precisely, $M^n$ can be only one of the following cases:
\begin{enumerate}
  \item [(a)] Cartan's example of isoparametric hypersurface $M^4$ in $S^5$ with four distinct principal curvatures;
  \item [(b)] the isoparametric hypersurface $M^6$ in $S^7$ with six distinct principal curvatures.
\end{enumerate}}
\vspace{2mm}
%\end{cor}

As the main result of this paper, we succeed in conquering the second difficulty and generalize de Almeida-Brito's Theorem 2 to any dimension, which provides us strong confidence in the Chern conjecture:

\begin{thm}\label{TY1}
Let $M^n$ $(n>3)$ be a closed $n$-dimensional Riemannian manifold. Suppose that $\mathfrak{a}$ is a smooth symmetric $(0,2)$ tensor field on $M^n$, and $\mathcal{A}$ is its dual tensor field of type $(1, 1)$. If the following conditions are satisfied:
%which is Codazzian and has $n$ distinct eigenvalues $\lambda_i$ ($i=1,\cdots, n$) everywhere, %and $A$ is the $(1,1)$ tensor field corresponding to $a$ via $g$.
%in addition, $\mathrm{tr}(a^k)$ $(k=1,\cdots, n-1)$ are constants.
\begin{enumerate}
  \item[(1.1)] $R_M\geq 0$;
  \item[(1.2)]  $\mathfrak{a}$ is Codazzian;
  \item[(1.3)] $\mathrm{tr}(\mathcal{A}^k)$ $(k=1,\ldots, n-1)$ are constants;
\end{enumerate}
then $\mathrm{tr}(\mathcal{A}^n)$ is a constant, and thus the eigenvalues of $\mathcal{A}$. 

Moreover, if $\mathcal{A}$ has $n$ distinct eigenvalues somewhere on $M^n$, then  $R_M\equiv 0$.
\end{thm}

Again, taking $\mathfrak{a}$ as the second fundamental form, we immediately obtain the following corollary which generalized the Corollary of Tang-Wei-Yan:

\begin{cor}\label{TY2}
Let $M^n$ $(n>3)$ be a closed hypersurface in the unit sphere $S^{n+1}$. If the following conditions are satisfied:
\begin{enumerate}
  \item[(2.1)] $R_M\geq 0$;
  \item[(2.2)] $\sum\limits_{i=1}^n\lambda_i^k$ $(k=1,\ldots, n-1)$ are constants for principal curvatures $\lambda_1,\ldots,\lambda_n$;
\end{enumerate}
then $M^n$ is isoparametric. 

Moreover, if $M^n$ has $n$ distinct principal curvatures somewhere, then $R_M\equiv 0$. 
\end{cor}

\begin{rem}
One can see that the result of \cite{SVW12}, i.e., Theorem 6 of \cite{SWY12} is a special case in dimension $n=6$ of Corollary \ref{TY2}.
\end{rem}

It is important to remark that condition $(2.1)$ doesn't  force us to eliminate any isoparametric hypersurfaces at all, since it is fulfilled by all the isoparametric hypersurfaces in spheres
via the following proposition, which generalizes Peng-Terng's Corollary 1 in \cite{PT83} for minimal isoparametric hypersurfaces:

\begin{prop}\label{TY3}
For any isoparametric hypersurface $M^n\subset S^{n+1}(1)$, the scalar curvature $R_M\geq 0$.
\end{prop}

We also remark that one can apply Theorem \ref{TY1} to other Codazzian symmetric $(0, 2)$ tensor $\mathfrak{a}$. For example, the manifolds with Codazzian Ricci tensor are the $\mathcal{B}$-manifolds defined by A. Gray (\cite{TY15}), which are also widely studied. 

The paper is organized as follows. In Section 2, we will first prove Proposition \ref{TY3} in order to get familiar with isoparametric hypersurfaces. In Section 3, we give some preliminaries for the proof of Theorem \ref{TY1} and in Section 4, we will finish the proof of Theorem \ref{TY1}. Then the Corollary \ref{TY2} follows at once.

%----------------------------------------------------------------------------------------------------------------------------------------------------------------------------------------
\section{\textbf{Scalar curvature of isoparametric hypersurfaces in spheres}}
We first list two equalities which will be useful later:

\begin{lem}\label{cot}
For any $\theta$ in the domain of definition,
\begin{eqnarray}
\sum_{k=1}^n\cot\left(\theta+\frac{k-1}{n}\pi\right)&=& n\cot n\theta;\label{cot1}\\
\sum_{k=1}^n\cot^2\left(\theta+\frac{k-1}{n}\pi\right)&=& n^2\cot^2 n\theta+n^2-n.\label{cot2}
\end{eqnarray}
\end{lem}

\begin{proof}
The proof is based on Milnor's paper \cite{Mil82}.
Substituting $z=e^{-2i\theta}$ into the equation 
$$|z^n-1|^2=\prod\limits_{k=1}^n|z-e^{-2i\frac{k-1}{n}\pi}|^2,$$
we obtain that
$$2^2\sin^2n\theta=2^{2n}\prod\limits_{k=1}^n\sin^2\left(\theta+\frac{k-1}{n}\pi\right).$$
Thus for $\theta\in(0,\frac{1}{n}\pi)$,  we get the trigonometric identity
\begin{equation}\label{sin}
\prod\limits_{k=1}^n\sin\left(\theta+\frac{k-1}{n}\pi\right)=2^{1-n}\sin n\theta.
\end{equation}
Then the general case follows by analytic continuation.

Taking derivatives on both sides of (\ref{sin}) and dividing the derivatives by two sides of (\ref{sin}) seperately, we obtain (\ref{cot1}).
Then (\ref{cot2}) follows easily.

\end{proof}

Now we give a proof of Proposition \ref{TY3}.
\begin{proof}
According to the fundamental result of M\"unzner, for an isoparametric hypersurface $M^n$ in $S^{n+1}(1)$, its principal curvatures could be written as $$\cot\theta, \cot\left(\theta+\frac{\pi}{g}\right),\cdots,\cot\left(\theta+\frac{g-1}{g}\pi\right)$$% \quad(\theta\in(0,\frac{1}{g}\pi)$$ 
with multiplicities $m_1, m_2,\cdots, m_g$ and $m_k=m_{k+2}$ with subscripts mod $g$.
Thus the mean curvature of $M^n$ is 
$$H=\sum_{i=1}^gm_i\cot\left(\theta+\frac{i-1}{g}\pi\right)$$
and the squared norm of the second fundamental form is 
$$S=\sum_{i=1}^gm_i\cot^2\left(\theta+\frac{i-1}{g}\pi\right).$$

If all the multiplicities are equal, denote $m_1=m_2=\cdots=m_g=m$, then $n=mg$ and the scalar curvature is (by using Lemma 2.1)
\begin{eqnarray*}
R_M&=&n(n-1)+H^2-S\\
&=& n(n-1)+\left(m\sum_{i=1}^g\cot(\theta+\frac{i-1}{g}\pi)\right)^2-m\sum_{i=1}^g\cot^2\left(\theta+\frac{i-1}{g}\pi\right)\\
&=& n(n-1)+m^2g^2\cot^2g\theta-mg^2\cot^2g\theta-mg^2+mg\\
&=&n(n-g)\left(1+\cot^2g\theta\right)\\
&\geq&0,
\end{eqnarray*}
and the ``=" holds if and only if $g=n$.

If all of the multiplicities are not equal, then $g$ is even and $m_1=m_3=\cdots=m_{g-1}$, $m_2=m_4=\cdots=m_g$.
Notice that now we get $n=\frac{g(m_1+m_2)}{2}$ and
\begin{eqnarray*}
H&=&m_1\sum_{i=1}^{\frac{g}{2}}\cot\left(\theta+\frac{2(i-1)}{g}\pi\right)+m_2\sum_{i=1}^{\frac{g}{2}}\cot\left(\theta+\frac{\pi}{g}+\frac{2(i-1)}{g}\pi\right)\\
&=&\frac{g}{2}m_1\cot\frac{g}{2}\theta+\frac{g}{2}m_2\cot\frac{g}{2}\left(\theta+\frac{\pi}{g}\right)\\
&=&\frac{g}{2}\left(m_1t-\frac{m_2}{t}\right),
\end{eqnarray*}
where in the last equality we use the notation $t:=\cot\frac{g}{2}\theta$ for convenience. 
\begin{eqnarray*}
S&=&m_1\sum_{i=1}^{\frac{g}{2}}\cot^2\left(\theta+\frac{2(i-1)}{g}\pi\right)+m_2\sum_{i=1}^{\frac{g}{2}}\cot^2\left(\theta+\frac{\pi}{g}+\frac{2(i-1)}{g}\pi\right)\\
&=&m_1\left((\frac{g}{2})^2\cot^2\frac{g}{2}\theta+(\frac{g}{2})^2-\frac{g}{2}\right)+m_2\left((\frac{g}{2})^2\cot^2\frac{g}{2}(\theta+\frac{\pi}{g})+(\frac{g}{2})^2-\frac{g}{2}\right) \\
&=&\frac{g^2}{4}\left(m_1t^2+\frac{m_2}{t^2}\right)+n\left(\frac{g}{2}-1\right).
\end{eqnarray*}
Thus the scalar curvature is
\begin{eqnarray*}
R_M&=&n(n-1)+H^2-S\\
&=& n(n-1)+\frac{g^2}{4}\left(m_1t-\frac{m_2}{t}\right)^2-\frac{g^2}{4}\left(m_1t^2+\frac{m_2}{t^2}\right)-n\left(\frac{g}{2}-1\right)\\
&=& \frac{g^2}{4}(m_1+m_2)(m_1+m_2-1)+\frac{g^2}{4}\left(m_1(m_1-1)t^2+m_2(m_2-1)\frac{1}{t^2}-2m_2m_2\right)\\
&=&\frac{g^2}{4}\left(m_1(m_1-1)(1+t^2)+m_2(m_2-1)(1+\frac{1}{t^2})\right)\\
&\geq&0,
\end{eqnarray*}
and the ``=" holds if and only if $m_1=m_2=1$.
\end{proof}

%----------------------------------------------------------------------------------------------------------------------------------------------------------------------------------------
\section{\textbf{Preliminaries for the proof of Theorem \ref{TY1}}}

 From now on, we assume that $M^n$ is connected and oriented. Otherwise, we can discuss on each connected component of $M^n$ or on the double covering of $M^n$.
 
 \subsection{Notations.}\quad
For convenience, we first make some notations. Let us denote by $\lambda_1(p)\leq\lambda_2(p)\leq\cdots\leq\lambda_n(p)$ the eigenvalues of $\mathcal{A}(p)$ for each $p\in M^n$. Note that $\lambda_i$ is continuous for each $i=1,\cdots,n.$  Rewrite condition $(1.3)$ as
 \begin{equation}\label{eqns}
\left\{ \begin{array}{llll}
\lambda_1+\cdots+\lambda_n=c_1\\
\lambda_1^2+\cdots+\lambda_n^2=c_2\\
\qquad\cdots\cdots\\
\lambda_1^{n-1}+\cdots+\lambda_n^{n-1}=c_{n-1},
\end{array}\right.
\end{equation}
where $c_1,\cdots,c_{n-1}$ are constants and define a function on $M^n$,
\begin{equation}\label{f}
f:=f(\lambda_1(p),\cdots,\lambda_n(p)):=\lambda_1^{n}+\cdots+\lambda_n^{n}.
\end{equation}
Notice that $f$ is a smooth function on $M^n$. Denote $\lambda=(\lambda_1,\lambda_2,\cdots,\lambda_n)$, sometimes we will just regard $f$ as a function of $\lambda$: $f=f(\lambda)$. It is obvious that $f$ is constant if and only if $\lambda_1,\cdots,\lambda_n$ are constants.

The following characteristic polynomial of $\mathcal{A}$ is important in our discussion:
\begin{equation}\label{F}
F(x)=\prod\limits_{i=1}^n(x-\lambda_i)=x^n-d_1x^{n-1}+d_2x^{n-2}-\cdots+(-1)^{n-1}d_{n-1}x+(-1)^nd_n,
\end{equation}
where
 \begin{equation}\label{di}
\left\{ \begin{array}{llll}
d_1=\lambda_1+\cdots+\lambda_n=\sum\limits_{i=1}^n\lambda_i\\
d_2=\lambda_1\lambda_2+\cdots+\lambda_{n-1}\lambda_n=\sum\limits_{i,j=1;\, i<j}^n\lambda_i\lambda_j\\
\qquad\cdots\cdots\\
d_{n-1}=\sum\limits_{i_1,\cdots, i_{n-1}=1;\, i_1<\cdots<i_{n-1}}^n\lambda_{i_1}\cdots\lambda_{i_{n-1}}\\
d_n=\lambda_1\cdots\lambda_n
\end{array}\right.
\end{equation}
are the elementary symmetric polynomial of $\lambda_1,\cdots,\lambda_n$.

By Newton's formula, $d_1,\cdots,d_{n-1}$ are determined uniquely by $c_1,\cdots,c_{n-1}$, thus are constants. Moreover,
\begin{equation}\label{dn}
d_n=d_n(f)=\frac{(-1)^{n-1}}{n}f+C,
\end{equation}
where $C$ is a constant depending only on $c_1,\cdots,c_{n-1}$.

We set
\begin{equation}\label{F0}
F_0(x)=x^n-d_1x^{n-1}+d_2x^{n-2}-\cdots+(-1)^{n-1}d_{n-1}x.
\end{equation}
Clearly, $F_0(x)$ is a polynomial of degree $n$ with coefficients depending on $c_1,\cdots,c_{n-1}$ and independent of $f$. Moreover, combining with (\ref{dn}), we have
\begin{equation}\label{F F0}
F(x)=F_0(x)-\frac{1}{n}f+(-1)^nC.
\end{equation}
\vspace{2mm}

\subsection{Discussion on $\Omega$.}\quad
Define a domain of $M^n$
\begin{eqnarray}
\Omega:=\big\{~~p\in M^n~|&&\sum_{i=1}^n\lambda_i^j(p)=c_j,~~ \forall ~j=1,\cdots,n-1,~\label{omega}\\
&& \textup{and} ~\lambda_1(p)<\lambda_2(p)<\cdots<\lambda_n(p)~\big\}.\nonumber
\end{eqnarray}
Then we will prove Theorem \ref{TY1} in the following two cases:

\vspace{2mm}
\subsubsection{}
\noindent
\textbf{Case 1: $\Omega=\varnothing$.}\quad
Rewrite $(\lambda_1,\cdots,\lambda_n)$ as 
$$(\lambda_1,\cdots,\lambda_n)=(\mu_1,\cdots,\mu_1,\cdots,\mu_g,\cdots,\mu_g),$$
with $\mu_1<\mu_2<\cdots<\mu_g$ of multiplicities $m_1, m_2,\cdots, m_g$, and $\displaystyle\sum_{k=1}^gm_k=n$. 
In this case, $g<n$, or equivalently, there is some $k\in\{1,\cdots,g\}$ with $m_k\geq 2$.

We will deal with this case by the following simple lemma which is totally in linear algebra:
\begin{lem}\label{mk 2}
There is at most one solution $(\mu_1,\cdots,\mu_g)$ to the equations 
 \begin{equation}\label{eqns}
\left\{ \begin{array}{llll}
m_1\mu_1+\cdots+m_g\mu_g=c_1\\
m_1\mu_1^2+\cdots+m_g\mu_g^2=c_2\\
\qquad\cdots\cdots\\
m_1\mu_1^{n-1}+\cdots+m_g\mu_g^{n-1}=c_{n-1}
\end{array}\right.
\end{equation}
with some $m_k\geq 2.$
\end{lem}

\begin{proof}
For convenience, rewrite the characteristic polynomial $F(x)$ in (\ref{F}) as
$$F(x)=\prod\limits_{i=1}^g(x-\mu_i)^{m_i}.$$
Notice that $F(\mu_1)=F(\mu_2)=\cdots=F(\mu_g)=0$. By Rolle's theorem, there exist $\tau_1,\cdots,\tau_{g-1}$ with 
$\mu_1<\tau_1<\mu_2<\tau_2<\mu_3<\cdots<\tau_{g-1}<\mu_g$ such that
$F'(\tau_1)=F'(\tau_2)=\cdots=F'(\tau_{g-1})=0$. Furthermore, noticing that $\mu_1<\tau_1<\mu_2<\tau_2<\mu_3<\cdots<\tau_{g-1}<\mu_g$ are all the possible roots of $F'(x)$, we see easily
$$F'(x)=n(x-\mu_1)^{m_1-1}(x-\tau_1)(x-\mu_2)^{m_2-1}(x-\tau_2)\cdots(x-\mu_g)^{m_g-1}.$$

Let $k$ be the positive number such that the multiplicity $m_k\geq 2$ and $m_i=1$ for $i\leq k-1$. Clearly, $\mu_k$ is the $k$-th root of $F'(x)$ and is uniquely determined by $F'(x)=F_0'(x)$, which is independent of $f$.

On the other hand, from 
$$F(\mu_k)=F_0(\mu_k)+(-1)^n d_n=0,$$
it follows that $d_n=(-1)^{n-1}F_0(\mu_k)$ is also uniquely determined.

Therefore, the polynomial $F(x)$ is uniquely determined, and thus $\mu_1,\cdots,\mu_g$ the real roots of $F(x)$.
\end{proof}

Given a sequence $m_1,\cdots,m_g$ with some $m_k\geq 2$, it follows from Lemma \ref{mk 2} that $f=(-1)^{n-1}n(d_n-C)$ is a uniquely determined constant function.
Since $m_i$ $(i=1,\cdots,g)$ are positive integers and $\sum_{i=1}^gm_i=n$, we have only finite cases with some $m_k\geq 2$, and in each case, $f$ is a uniquely determined constant. Thus the set of possible values of $f$ is a discrete set. However, as we mentioned before, $f$ is a smooth function on $M^n$. Therefore, $f$ must be constant on $M^n$, and thus the eigenvalues $\lambda_1,\cdots,\lambda_n$ of $\mathcal{A}$.

\vspace{3mm}

\subsubsection{}
\noindent
\textbf{Case 2: $\Omega\neq \varnothing$.}\quad At the first glance, $f$ is a smooth function on closed $M^n$, thus the range of function $f$ is $\textup{Im} f=[a_0, b_0]$ $(a_0\leq b_0)$.
Our first task in this case is to investigate $\textup{Im} f$. 

Since  $\Omega\neq \varnothing$, it is directly seen that the polynomial $F(x)$ defined in 
(\ref{F}) has $n$ distinct real roots on $\Omega$. Equivalently, the equation 
\begin{equation}\label{F F0 f}
F(x)=\prod\limits_{i=1}^n(x-\lambda_i)=F_0(x)-\frac{1}{n}f+(-1)^nC=0
\end{equation}
has $n$ distinct roots $\lambda_1<\cdots<\lambda_n$ on $\Omega$. To determine $\textup{Im} f$, we start from examining the polynomial $F_0(x)$ which is independent of $f$.

Notice that 
$$F(\lambda_1)=F(\lambda_2)=\cdots=F(\lambda_n)=0.$$
By Rolle's theorem, we can find $\tau_1,\cdots,\tau_{n-1}$ with $\lambda_1<\tau_1<\lambda_2<\tau_2<\cdots<\tau_{n-1}<\lambda_n$ such that
$$F_0'(\tau_i)=F'(\tau_i)=0,\quad\forall ~~i=1,\cdots,n-1.$$
Thus for the polynomial function $F_0(x)$ of degree $n$, $\tau_1,\cdots,\tau_{n-1}$ are all the extreme points of $F_0(x)$. We define $b'$ to be the minimum of all the local maximal values of $F_0(x)$, and $a'$ to be the maximum of all the local minimum values. To be more precise, when $n$ is odd, 
\begin{equation}\label{b' for n odd}
b':=\textup{min}\{F_0(\tau_1), F_0(\tau_3),\cdots, F_0(\tau_{n-2})\},
\end{equation}
and
\begin{equation}\label{a'}
a':=\textup{max}\{F_0(\tau_2), F_0(\tau_4),\cdots,F_0(\tau_{n-1})\}.
\end{equation}
For example, we give a figure of polynomial function $F_0(x)$ with degree $5$ as follows:
\begin{figure}[h]
\centering
\includegraphics[width=12cm,height=6cm]{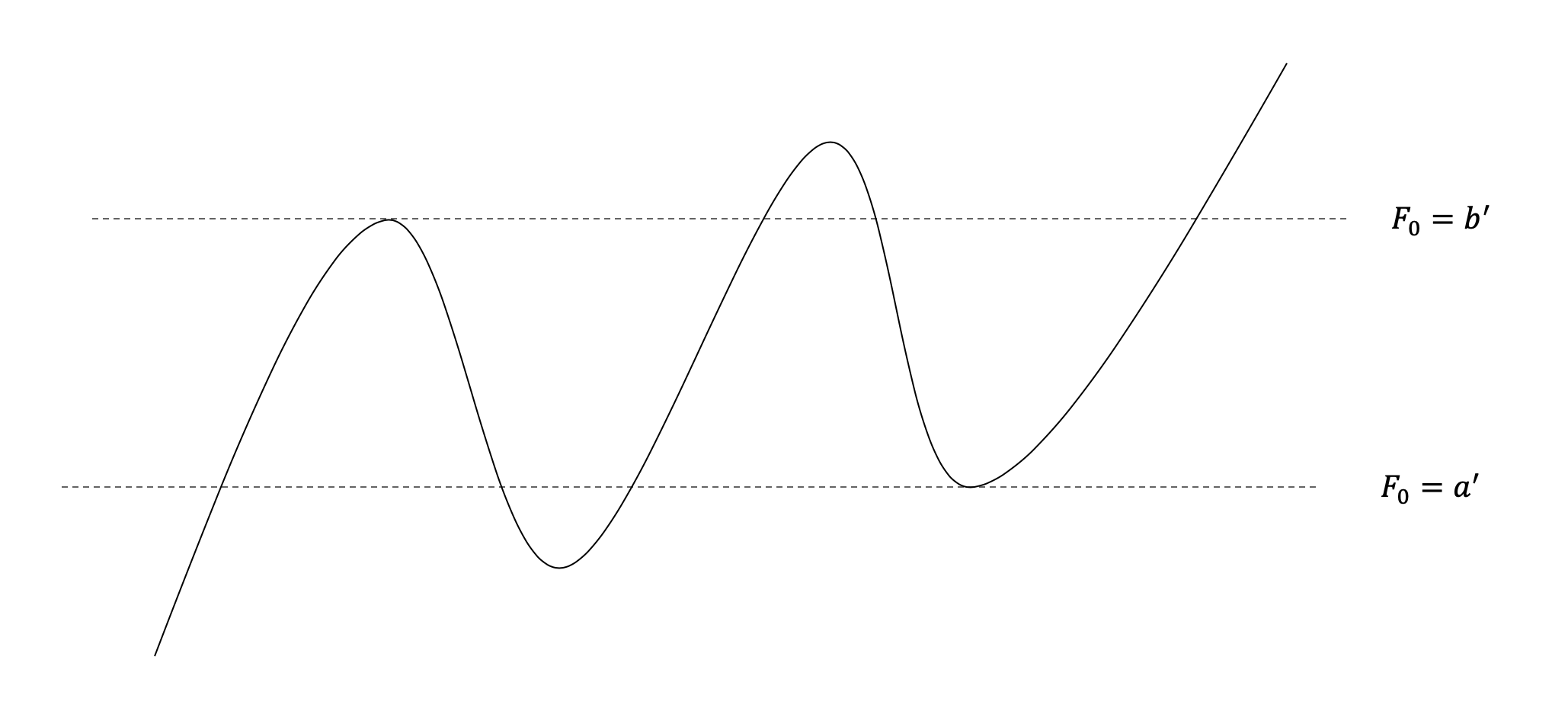}
\caption{}
\end{figure}
\newpage
When $n$ is even,
\begin{equation}\label{b' for n even}
b':=\textup{min}\{F_0(\tau_2), F_0(\tau_4),\cdots,F_0(\tau_{n-2})\},
\end{equation}
and 
\begin{equation}\label{a' for n even}
a':=\textup{max}\{F_0(\tau_1), F_0(\tau_3),\cdots, F_0(\tau_{n-1})\},
\end{equation}
For example, we give a figure of polynomial function $F_0(x)$ with degree $4$ as follows:
\begin{figure}[h]
\centering
\includegraphics[width=10cm,height=5cm]{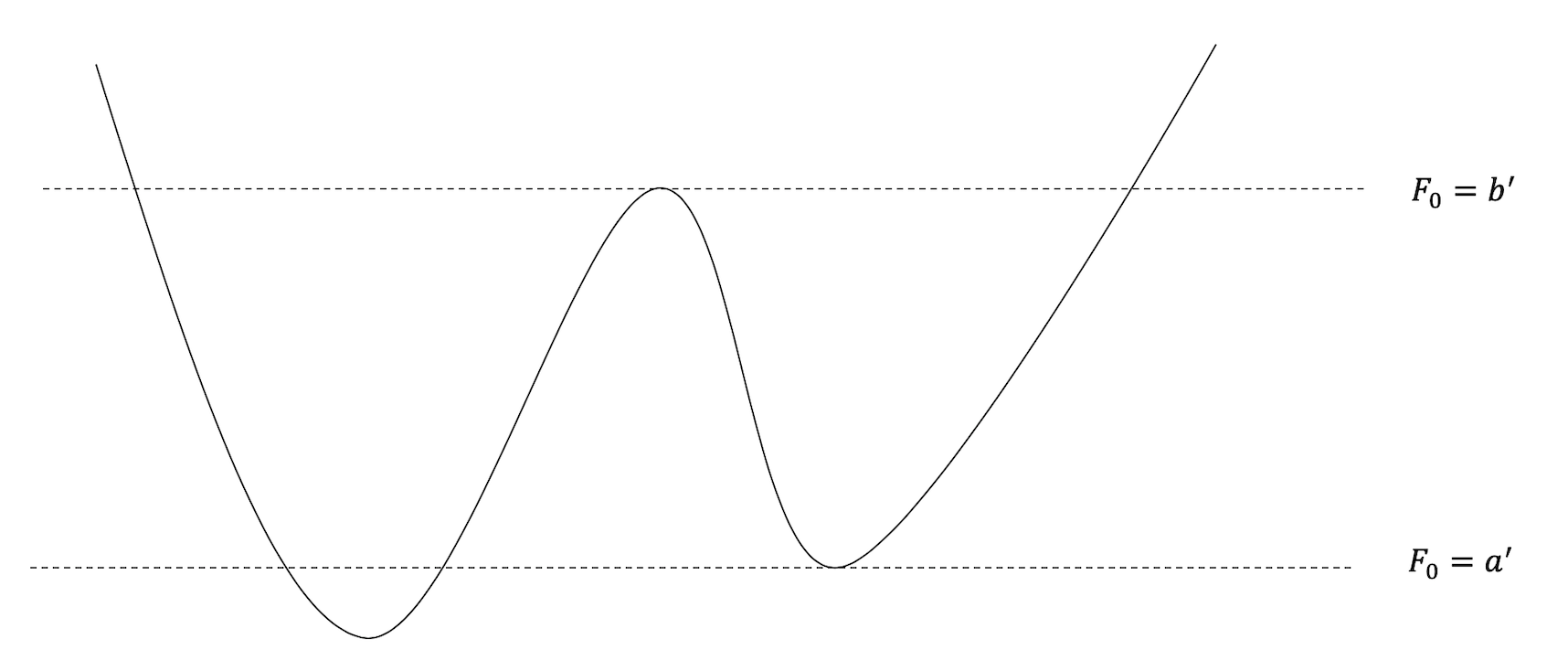}
\caption{}
\end{figure}

Recall the facts we mentioned above, that is, the equation (\ref{F F0 f}) has $n$ distinct roots on $\Omega$, it follows directly that
$$b'>a'.$$
Moreover, Observe that for any $\xi\in[a_0, b_0]$, $\xi$ could be expressed as the $n$-th power sum of the roots of equations (\ref{eqns}), that is, the equation
$$F_0(x)-\frac{1}{n}\xi+(-1)^nC=0$$
has $n$ real roots. Therefore, defining
\begin{equation}\label{ab}
b:=n\left( ~b'+(-1)^nC~\right),\quad a:=n\left( ~a'+(-1)^nC~\right),
\end{equation}
we obtain immediately that
\begin{equation*}\label{bb'}
\textup{Im} f=[a_0, b_0]\subset [a, b].
\end{equation*}
So we need to consider four cases: (1) $a_0>a$, $b_0<b$; (2) $a_0=a$, $b_0<b$;
(3) $a_0>a$,  $b_0=b$; (4) $a_0=a$, $b_0=b$.

For the case $(1)$, that is, all the eigenvalues of $\mathcal{A}$ are distinct on $M^n$, it is already completed by \cite{TWY20}. In the following, it is sufficient for us to deal with the case (4), the other cases are verbatim.

From now on, we assume that $\textup{Im} f=[a, b]$ with $a<b$, and $a, b$ are achieved on the points where $F_0(x)$ achieves the maximum of its local minimum values and the minimum of its local maximum values, as we illustrated in (\ref{b' for n odd})-(\ref{a' for n even}) and (\ref{ab}).

Using the same notations with those in \cite{dB90}, we define 
\begin{eqnarray}\label{XYZ}
X&:=& \{p\in M^n:~~f(p)=a\}=f^{-1}(a)\nonumber\\
Y&:=&\{p\in M^n:~~a<f(p)<b\}\\
Z&:=&\{p\in M^n:~~f(p)=b\}=f^{-1}(b) \nonumber
\end{eqnarray}
Obviously,
$$M^n=X\cup Y\cup Z.$$
If $Y=\varnothing$, then $f=a$ or $b$, and Theorem \ref{TY1} follows from the continuity of $f$. 

From now on, we will assume $Y\neq\varnothing$. 

From the discussion on $\textup{Im} f=[a_0, b_0]$ as above and the assumption $[a_0, b_0]=[a, b]$, we derive the following geometric illustration of $Y$ directly:

\begin{lem}\label{Y}
Let $(\lambda_1,\cdots,\lambda_n)$ be a solution to the equations (\ref{eqns}). Suppose $\Omega\neq\varnothing$ and $\textup{Im} f=[a, b]$.
Then 
$$Y\subset \Omega.$$\end{lem}

For $0<\varepsilon<\frac{b-a}{2}$, we let (similar to that in \cite{dB90})
\begin{eqnarray}\label{XYZvarepsilon}
X_{\varepsilon}&:=& \{p\in M^n:~~a<f(p)<a+\varepsilon\}\nonumber
\\
Y_{\varepsilon}&:=& \{p\in M^n:~~a+\varepsilon \leq f(p)\leq b-\varepsilon\}\\
Z_{\varepsilon}&:=& \{p\in M^n:~~b-\varepsilon<f(p)<b\} \nonumber
\end{eqnarray}
and infer from (\ref{XYZ}) that
$$Y=X_{\varepsilon}\cup Y_{\varepsilon}\cup Z_{\varepsilon}.$$
\vspace{4mm}

\section{\textbf{Proof of Theorem \ref{TY1}}}
In this section, we continue to deal with the Case 2 in Subsection 3.2.2 under the assumption that $M^n$ is closed, connected and oriented, $\Omega\neq\varnothing$ and $\textup{Im} f=[a, b]$ $(a<b)$.

\subsection{Structure equations on $Y$.}\quad

Firstly, we will take a look at the structure equations on the open set $Y$ of $M^n$.
Actually, we are going to repeat some definitions and calculations of \cite{TWY20} in this subsection.

Locally, we choose an oriented orthonormal frame fields $\{e_i, ~i=1,\ldots,n\}$ on $M^n$. Let $\{\theta_i, ~i=1,\ldots,n\}$ be the dual frame. Then one has the structure equations:
\begin{equation*}\label{str eqn}
\left\{ \begin{array}{ll}
d\theta_i=\sum\limits_{j=1}^n\omega_{ij}\wedge\theta_j\\
d\omega_{ij}=\sum\limits_{k=1}^n\omega_{ik}\wedge\omega_{kj}-R_{ij},
\end{array}\right.
\end{equation*}
where $\omega_{ij}$ is the connection form %defined by $\omega_{ij}(e_k):=\langle \nabla_{e_k}e_i, e_j\rangle$
and $R_{ij}=\frac{1}{2}\sum\limits_{k,l=1}^nR_{ijkl}\theta_k\wedge\theta_l$ is the curvature form.

Let $\mathfrak{a}$ be a smooth symmetric $(0,2)$ tensor, which can be denoted by $\mathfrak{a}=\sum\limits_{i,j=1}^n\mathfrak{a}_{ij}\theta_i\otimes\theta_j$, where $\mathfrak{a}_{ij}=\mathfrak{a}(e_i, e_j)$ is smooth and $\mathfrak{a}_{ij}=\mathfrak{a}_{ji}$. Then the covariant derivative of $\mathfrak{a}$ can be written by
\begin{equation*}
  \nabla \mathfrak{a}=\sum_{i,j,k=1}^n \mathfrak{a}_{ijk}\theta_i\otimes\theta_j\otimes\theta_k,
\end{equation*}
where \begin{equation}\label{2}
\sum_{k=1}^n\mathfrak{a}_{ijk}\theta_k=d\mathfrak{a}_{ij}+\sum_{m=1}^n(\mathfrak{a}_{im}\omega_{mj}+\mathfrak{a}_{mj}\omega_{mi}).
\end{equation}
In addition, according to the assumption that the tensor $\mathfrak{a}$ is Codazzian, that is, $(\nabla_{e_k} \mathfrak{a})(e_i, e_j)=(\nabla_{e_i} \mathfrak{a})(e_k, e_j)$ for any $i,j,k=1,\ldots,n$.
It implies immediately that $\mathfrak{a}_{ijk}$ is symmetric, and so is $\nabla \mathfrak{a}$. %In addition, suppose $a$ has $n$ distinct eigenvalues everywhere, say $\lambda_i$ ($i=1,\ldots, n$).

Next, we choose a proper coordinate system on $Y$ such that $(U, (\theta_1,\ldots,\theta_n))$ is admissible (\cite{dB90}). Namely, $(U, (\theta_1,\ldots,\theta_n))$ satisfies
\begin{itemize}
  \item $(\theta_1,\ldots,\theta_n)$ is a smooth orthonormal coframe field on an open subset $U$ of $Y$;
  \item $\theta_1\wedge\ldots\wedge\theta_n=\texttt{vol}$, the volume form on $U$;
  \item $\mathfrak{a}=\sum\limits_{i=1}^n\lambda_i\theta_i\otimes\theta_i$.
\end{itemize}
Evidently, when $(U, (\theta_1,\ldots,\theta_n))$ is admissible, the connection forms $\omega_{ij}$ on U are uniquely determined and $\mathfrak{a}_{ij}=\lambda_i\delta_{ij}$.

On the other hand, it follows from Lemma \ref{Y} that $Y\subset\Omega$. Thus each $\lambda_i$ ($i=1,\ldots,n$) is 
%simple, it follows that $\lambda_i$ ($i=1,\ldots,n$) is 
smooth on $Y$.
We differentiate it to get the smooth $1$-form $d\lambda_i$, which can be expressed by the metric form $\theta_k$ as
\begin{equation*}\label{5}
d\lambda_i=\sum\limits_{j=1}^n\lambda_{ij}\theta_j,
\end{equation*}
where $\lambda_{ij}$ are smooth functions on $M^n$.
Besides, express the connection form $\omega_{ij}$ as
\begin{equation}\label{beta1}
\omega_{ij}:=\sum\limits_{k=1}^n\beta_{ijk}\theta_k
\end{equation}
where $\beta_{ijk}=\omega_{ij}(e_k)$.
Then it follows from equation (\ref{2}) immediately that
\begin{equation*}\label{3}
\sum_{k=1}^n\mathfrak{a}_{iik}\theta_k = d\lambda_i=\sum\limits_{k=1}^n\lambda_{ik}\theta_k, ~~\forall~i=1,\ldots,n
\end{equation*}
\begin{equation*}\label{4}
\sum_{k=1}^n \mathfrak{a}_{ijk}\theta_k = (\lambda_i-\lambda_j)\omega_{ij}=(\lambda_i-\lambda_j)\sum\limits_{k=1}^n\beta_{ijk}\theta_k, ~~\forall~ i\neq j.
\end{equation*}
Equivalently,
\begin{eqnarray}
a_{iik} &=& \lambda_{ik} \label{6} \\
a_{ijk} &=& (\lambda_i-\lambda_j)\beta_{ijk}, ~~\forall~ i\neq j. \label{7}
\end{eqnarray}

Differentiating the equations
\begin{equation}\label{9}
\left\{ \begin{array}{llll}
\lambda_1+\cdots+\lambda_n=c_1\\
\lambda_1^2+\cdots+\lambda_n^2=c_2\\
\qquad\ldots\cdots\\
\lambda_1^{n-1}+\cdots+\lambda_n^{n-1}=c_{n-1}\\
\lambda_1^{n}+\cdots+\lambda_n^{n}=f,
\end{array}\right.
\end{equation}
we obtain for each $j=1,\ldots,n$,
\begin{equation}\label{10}
\begin{pmatrix}
1 & 1 & \cdots & 1\\
\lambda_1 & \lambda_2 &\cdots & \lambda_n\\
\vdots & \vdots &\ddots &\vdots\\
\lambda_1^{n-2} & \lambda_2^{n-2} &\cdots & \lambda_n^{n-2}\\
\lambda_1^{n-1} & \lambda_2^{n-1} &\cdots & \lambda_n^{n-1}
\end{pmatrix}
\begin{pmatrix}
\lambda_{1j}\\
\lambda_{2j}\\
\vdots\\
\lambda_{n-1, j}\\
\lambda_{nj}%\vphantom{\frac{J}{\sqrt{E_1^2+J^2}}} 0
\end{pmatrix}
=
\begin{pmatrix}
0\\
0\\
\vdots\\
0\\
f_j/n%\vphantom{\frac{J}{\sqrt{E_1^2+J^2}}} 0
\end{pmatrix},
\end{equation}
where $f_j$ is defined as follows:
\begin{equation*}\label{8}
df=\sum_{j=1}^n(\sum_{i=1}^nn\lambda_i^{n-1}\lambda_{ij})\theta_j:=\sum_{j=1}^nf_j\theta_j.
\end{equation*}
Denote the $n\times n$ Vandermonde matrix on the left hand of (\ref{10}) by $D$. It is known that its determinant
\begin{equation*}\label{11}
\gamma:=\det D=\prod_{
k,l=1;~k>l}^n(\lambda_k-\lambda_l)\neq 0.
\end{equation*}
Then it follows from the equations (\ref{10}) that
\begin{eqnarray}\label{12}
\lambda_{ij}&=&(-1)^{i+n}\frac{f_j}{n\gamma}\prod_{k,l=1;~k,l\neq i;~k>l
%\mbox{\tiny$\begin{array}{c}
%k,l=1\\
%k,l\neq i; k>l\end{array}$}
}^n(\lambda_k-\lambda_l)\nonumber\\
&=&(-1)^{n+1}\frac{f_j}{n}\cdot\frac{1}{\prod\limits_{k=1;~
k\neq i}^n(\lambda_k-\lambda_i)}.
\end{eqnarray}

\vspace{3mm}

\subsection{The $(n-1)$-form  $\psi$.}\quad
As in \cite{TWY20}, we define an $(n-1)$-form $\psi$ as follows:
\begin{equation*}\label{psi}
\psi=\sum_{\sigma}S(\sigma)\theta_{i_1}\wedge\theta_{i_2}\wedge\cdots\wedge\theta_{i_{n-2}}\wedge\omega_{i_{n-1}i_n},
\end{equation*}
where $\sigma(1,\ldots,n)=(i_1,\ldots,i_n)$ is a permutation and $S(\sigma)$ is the sign of $\sigma$. By Lemma 4.1. of \cite{TWY20}, we know that $\psi$ is globally well defined on $Y$.

From \cite{TWY20}, we also have the differential of $\psi$ as follows:
\begin{equation}\label{dpsi}
(-1)^{n+1} d\psi = \left((n-2)!\cdot R_M+\frac{(n-3)!}{n^2}\sum_{r=1}^n (-L(r))f_r^2\right)\cdot\texttt{vol}\quad\quad~\textup{on}~~Y,
\end{equation}
where $\texttt{vol}$ is the volume form and $L(r)$ is defined as
\begin{equation*}
  L(r):=\sum_{p,q=1;~p\neq q;~ p,q\neq r
}^n \frac{1}{(\lambda_r-\lambda_p)(\lambda_r-\lambda_q)\cdot\prod\limits_{k=1;~k\neq p}^n(\lambda_k-\lambda_p)\cdot\prod\limits_{l=1;~l\neq q}^n(\lambda_l-\lambda_q)}.
\end{equation*}
By Lemma 2.1 ( a key lemma ) of \cite{TWY20}, we know that $L(r)<0$ for each $r=1,\cdots, n$. Then it follows from the assumption $R_M\geq 0$ that
\begin{equation}\label{int dpsi}
\int_Y(-1)^{n+1} d\psi \geq 0.
\end{equation}

Next, we are going to calculate $df\wedge\psi$.

From (\ref{beta1}) and (\ref{7}), it follows that
\begin{eqnarray}\label{df wedge psi}
&&df\wedge\psi\\
&=&\sum_{\sigma}S(\sigma)\left(\sum_{k=1}^nf_k\theta_k\right)\wedge\theta_{i_1}\wedge\theta_{i_2}\wedge\cdots\wedge\theta_{i_{n-2}}\wedge\omega_{i_{n-1}i_n} \nonumber\\
%&=&  \sum_{\sigma}\sum_{k, l=1}^nS(\sigma)f_k\beta_{i_{n-1}i_nl}\theta_k\wedge\theta_{i_1}\wedge\theta_{i_2}\wedge\cdots\wedge\theta_{i_{n-2}}\wedge\theta_l\nonumber
%\\
%&=&  \sum_{\sigma}S(\sigma)\Big(f_{i_{n-1}}\beta_{i_{n-1}i_ni_n}\theta_k\wedge\theta_{i_1}\wedge\theta_{i_2}\wedge\cdots\wedge\theta_{i_{n-2}}\wedge\theta_{i_n}\nonumber\\
%&& \qquad\qquad +f_{i_{n}}\beta_{i_{n-1}i_ni_{n-1}}\theta_{i_n}\wedge\theta_{i_1}\wedge\theta_{i_2}\wedge\cdots\wedge\theta_{i_{n-2}}\wedge\theta_{i_{n-1}}\Big)\nonumber\\
&=& (-1)^n  \sum_{\sigma}S(\sigma)\left(f_{i_{n-1}}\beta_{i_{n-1}i_ni_n}-f_{i_{n}}\beta_{i_{n-1}i_ni_{n-1}}\right)\theta_{i_1}\wedge\theta_{i_2}\wedge\cdots\wedge\theta_{i_{n}}.\nonumber
\end{eqnarray}
By (\ref{6}) and (\ref{7}), we have
\begin{eqnarray*}
&&f_{i_{n-1}}\beta_{i_{n-1}i_ni_n}-f_{i_{n}}\beta_{i_{n-1}i_ni_{n-1}}\\
&=&f_{i_{n-1}}\frac{\lambda_{i_ni_{n-1}}}{\lambda_{i_{n-1}}-\lambda_{i_n}}-f_{i_{n}}\frac{\lambda_{i_{n-1}i_n}}{\lambda_{i_{n-1}}-\lambda_{i_n}}  \\
%&=& (-1)^{n+1}\frac{1}{n(\lambda_{i_{n-1}}-\lambda_{i_n})}\Big(\frac{f_{i_{n-1}}^2}{\prod\limits_{k=1;~k\neq i_n}^n(\lambda_k-\lambda_{i_n})}-\frac{f_{i_{n}}^2}{\prod\limits_{k=1;~k\neq i_{n-1}}^n(\lambda_k-\lambda_{i_{n-1}})}\Big)   \\
&=& (-1)^{n+1}\frac{1}{n(\lambda_{i_{n-1}}-\lambda_{i_n})^2} \Big(\frac{f_{i_{n-1}}^2}{\prod\limits_{k=1;~k\neq i_{n-1}, i_n}^n(\lambda_k-\lambda_{i_n})}-\frac{f_{i_{n}}^2}{\prod\limits_{k=1;~k\neq i_{n-1}, i_n}^n(\lambda_k-\lambda_{i_{n-1}})}\Big).\end{eqnarray*}
%Furthermore, denoting by $\sigma'=\sigma\circ\left(\begin{array}{cc}
%i_{n-1}& i_n\\
%i_n& i_{n-1}
%\end{array}\right)$, we have
%\begin{eqnarray*}
%&&\sum_{\sigma}S(\sigma)\frac{f_{i_{n-1}}^2}{(\lambda_{i_{n-1}}-\lambda_{i_n})^2\prod\limits_{k=1;~k\neq i_{n-1}, i_n}^n(\lambda_k-\lambda_{i_n})}\theta_{i_1}\wedge\cdots\wedge\theta_{i_{n-2}}\wedge\theta_{i_{n-1}}\wedge\theta_{i_n}\\
%&=&\sum_{\sigma'}S(\sigma')\frac{-f_{i'_{n}}^2}{(\lambda_{i'_{n}}-\lambda_{i'_{n-1}})^2\prod\limits_{k=1;~k\neq i'_{n-1}, i'_n}^n(\lambda_k-\lambda_{i'_{n-1}})}\theta_{i'_1}\wedge\cdots\wedge\theta_{i'_{n-2}}\wedge\theta_{i'_{n-1}}\wedge\theta_{i'_n}\\
%&=&\sum_{\sigma}S(\sigma)\frac{f_{i_{n}}^2}{(\lambda_{i_{n-1}}-\lambda_{i_n})^2\prod\limits_{k=1;~k\neq i_{n-1}, i_n}^n(\lambda_k-\lambda_{i_{n-1}})}\theta_{i_1}\wedge\theta_{i_2}\wedge\cdots\wedge\theta_{i_{n}}.
%\end{eqnarray*}
Therefore, %we arrive at
\begin{eqnarray*}
df\wedge\psi
&=&-\frac{2}{n}\sum_{\sigma}S(\sigma)\frac{f_{i_{n}}^2}{(\lambda_{i_{n-1}}-\lambda_{i_n})^2\prod\limits_{k=1;~k\neq i_{n-1}, i_n}^n(\lambda_k-\lambda_{i_{n-1}})}\theta_{i_1}\wedge\theta_{i_2}\wedge\cdots\wedge\theta_{i_{n}}\\
&=&-\frac{2(n-2)!}{n}\sum_{i,j=1; i\neq j}^n\frac{f_i^2}{(\lambda_i-\lambda_j)^2\prod\limits_{k=1; k\neq i,j}^n(\lambda_k-\lambda_j)}\cdot \texttt{vol}\qquad\qquad \textup{on}~Y.
\end{eqnarray*}

For $i=1,\cdots,n$, 
define $$u_i:=-\frac{2(n-2)!}{n}\sum_{j=1; j\neq i}^n\frac{1}{(\lambda_i-\lambda_j)^2\prod\limits_{k=1; k\neq i,j}^n(\lambda_k-\lambda_j)},$$
thus
\begin{equation}\label{df}
df\wedge\psi=\sum_{i=1}^nu_if_i^2\cdot\texttt{vol}\qquad\qquad \textup{on}~Y.
\end{equation}

When $Y\neq\varnothing$, for any smooth function $\eta:(a, b)\longrightarrow \mathbb{R}$ with compact support, we follow \cite{dB90} to apply the Stokes theorem to 
$$d\left((\eta\circ f)\psi\right)=(\eta\circ f)d\psi+(\eta'\circ f)df\wedge\psi$$
to obtain
\begin{equation}\label{int on Y}
\int_Y(\eta\circ f)d\psi+\int_Y(\eta'\circ f)df\wedge\psi=0.
\end{equation}

Given a small $\varepsilon$, we choose a smooth function $\eta_{\varepsilon}:\mathbb{R}\rightarrow \mathbb{R}$ such that
\begin{itemize}
\item[(1)] $0\leq\eta_{\varepsilon}\leq 1$;
\item[(2)] $\eta_{\varepsilon}(t)=0$, for $a\leq t\leq a+\frac{\varepsilon}{n}$ or $b-\frac{\varepsilon}{n}\leq t\leq b$;
\item[(3)] $\eta_{\varepsilon}(t)=1$, for $a+\varepsilon\leq t\leq b-\varepsilon$;
\item[(4)] $\eta'_{\varepsilon}\geq 0$ on $(-\infty, \frac{a+b}{2})$,  $\eta'_{\varepsilon}\leq 0$ on $(\frac{a+b}{2}, +\infty)$.
\end{itemize}

It follows from (\ref{int dpsi}), (\ref{df}) and (\ref{int on Y}) that
\begin{eqnarray}\label{100}
0&\leq& \int_Y(-1)^{n+1}(\eta_{\varepsilon}\circ f)d\psi=\int_Y(-1)^n(\eta_{\varepsilon}'\circ f)df\wedge\psi\\
&=&\int_Y(\eta_{\varepsilon}'\circ f)\sum_{i=1}^n(-1)^nu_if_i^2\cdot\texttt{vol}\nonumber\\
&\leq& \int_YA\cdot|\eta_{\varepsilon}'\circ f|\cdot|df|^2\cdot\texttt{vol}\nonumber
\end{eqnarray}
where for the last inequality and the number $A$ we have used the following assertion whose proof is left to the end of this paper.

\vspace{2mm}

\noindent
\textbf{Assertion 4.1.} There exists a constant $A>0$ depending only on $n$ and $c_1,c_2,\cdots,c_{n-1}$, such that 
\begin{equation*}
(-1)^nu_i\geq -A ~~\textup{on} ~~Z_{\varepsilon}\quad \textup{and}\quad (-1)^nu_i\leq A ~~\textup{on} ~~X_{\varepsilon}.
\end{equation*}
\vspace{2mm}

On the other hand, for any smooth function $h: \mathbb{R}\rightarrow \mathbb{R}$, we may apply Stokes's theorem to 
$$d^*((h\circ f)df)=(h'\circ f)|df|^2\cdot\texttt{vol}+(h\circ f)\Delta f\cdot \texttt{vol}$$
to obtain 
\begin{equation}\label{101}
\int_{M^n}(h'\circ f)|df|^2\cdot\texttt{vol}+\int_{M^n}(h\circ f)\Delta f\cdot \texttt{vol}=0.
\end{equation}
Let $h_{\varepsilon}: \mathbb{R}\rightarrow \mathbb{R}$ be the smooth function given by
$$h_{\varepsilon}=\left\{\begin{array}{cc}
\eta_{\varepsilon}-1\quad \textup{on}~ (-\infty, \frac{a+b}{2}]\\
1-\eta_{\varepsilon}\quad \textup{on}~ [\frac{a+b}{2}, +\infty)
\end{array}\right.$$
Note that $h'_{\varepsilon}=|\eta'_{\varepsilon}|$. It follows from (\ref{101}) that 
\begin{equation*}
\int_{Y}|\eta_{\varepsilon}'\circ f|\cdot|df|^2\cdot\texttt{vol}=-\int_{M^n}(h_{\varepsilon}\circ f)\Delta f\cdot\texttt{vol}\leq\int_{M^n}|h_{\varepsilon}\circ f|\cdot|\Delta f|\cdot\texttt{vol}
\end{equation*}
By construction $|h_{\varepsilon}|\leq 1$ and $h_{\varepsilon}\circ f=0$ on $Y_{\varepsilon}$.

Next, we will use a generalized version of Lemma 1 in \cite{dB90}. Their proof is for $n=3$, but one can follow their proof and generalize the result to any dimension,
just modifying the tiny mistake in writing $\texttt{vol}\mathcal{V}<\frac{\epsilon_0}{4c}\texttt{vol} M$ to $\texttt{vol}\mathcal{V}<\frac{\epsilon_0}{4c}$:
\begin{equation*}
\lim_{\varepsilon\rightarrow 0}\int_{M^n-Y_{\varepsilon}}|\Delta f|\cdot\texttt{vol}=0\qquad\quad \textup{if}~ X\cup Z\neq \varnothing.
\end{equation*}
Then we obtain 
\begin{eqnarray*}
&&\lim_{\varepsilon\rightarrow 0}\int_{M^n}|h_{\varepsilon}\circ f|\cdot|\Delta f|\cdot\texttt{vol}\\
&=&\lim_{\varepsilon\rightarrow 0}\int_{M^n-Y_{\varepsilon}}|h_{\varepsilon}\circ f|\cdot|\Delta f|\cdot\texttt{vol}\\
&\leq&\lim_{\varepsilon\rightarrow 0}\int_{M^n-Y_{\varepsilon}}|\Delta f|\cdot\texttt{vol}\\
&=&0,
\end{eqnarray*}
and thus
$$\lim_{\varepsilon\rightarrow 0}\int_{Y}|\eta_{\varepsilon}'\circ f|\cdot|df|^2\cdot\texttt{vol}=0.$$
Combining with (\ref{100}), we have
$$\lim_{\varepsilon\rightarrow 0}\int_{Y}(-1)^{n+1}(\eta_{\varepsilon}\circ f)d\psi=0.$$

At last, since by assumption $R_M\geq 0$, Lemma 2.1 of \cite{TWY20} and (\ref{dpsi}) lead us to
$$0\leq\int_{Y_{\varepsilon'}}\frac{(n-3)!}{n^2}\sum_{r=1}^n (-L(r))f_r^2\cdot\texttt{vol}\leq \int_Y(-1)^{n+1}(\eta_{\varepsilon}\circ f) d\psi$$
for all $0<\varepsilon\leq \varepsilon'<\frac{b-a}{2}$, it follows that $f_r=0$ on $Y$ for any $r=1,\cdots,n$. Thus $f$ is constant on $Y$, and furthermore, constant on $M^n$.

The proof of Theorem \ref{TY1} is now complete.

\vspace{5mm}
We conclude this section by giving a proof of 
\vspace{2mm}

\noindent
\textbf{Assertion 4.1.} There exists a constant $A>0$ depending only on $n$ and $c_1,c_2,\cdots,c_{n-1}$, such that 
\begin{equation*}
(-1)^nu_i\geq -A ~~\textup{on} ~~Z_{\varepsilon}\quad \textup{and}\quad (-1)^nu_i\leq A ~~\textup{on} ~~X_{\varepsilon}.
\end{equation*}
\begin{proof}
Recall that $$u_i:=-\frac{2(n-2)!}{n}\sum_{j=1; j\neq i}^n\frac{1}{(\lambda_i-\lambda_j)^2\prod\limits_{k=1; k\neq i,j}^n(\lambda_k-\lambda_j)}.$$
For $j\neq i$, define 
\begin{equation}\label{uij} 
u_{ij}=\frac{1}{(\lambda_i-\lambda_j)^2\prod\limits_{k=1; k\neq i,j}^n(\lambda_k-\lambda_j)},
\end{equation}
then
\begin{equation}\label{ui} 
u_i=-\frac{2(n-2)!}{n}\sum_{j=1; j\neq i}^n u_{ij}.
\end{equation}

As the first step, for the function $f=f(\lambda(p))$ with $\lambda(p)=(\lambda_1(p),\cdots,\lambda_n(p))$, we need to clarify that according to the definition of $a', b'$ and $a, b$, when $f\rightarrow a$ from above or $f\rightarrow b$ from below, $\lambda_1<\lambda_2<\cdots<\lambda_n$ and all the $\lambda_i's$ depend on $f$ continuously. This is generally not right without the restriction on $a' $ and $b'$.

Denote 
\begin{eqnarray*}
&&f^{-1}(b)=(\beta_1,\beta_2,\cdots, \beta_n)\\
&& f^{-1}(a)=(\alpha_1,\alpha_2,\cdots,\alpha_n).
\end{eqnarray*}
By the definition of $a$ and $b$, we are surprised to find from Figure 1 and Figure 2 that when $f\rightarrow a$ or $f\rightarrow b$, the multiplicity of $\alpha_i$ or $\beta_j$ $(i, j=1,\cdots, n)$
is at most $2$. More precisely, when $n$ is odd, 
\begin{eqnarray}\label{n odd}
&&\beta_1\leq\beta_2<\beta_3\leq\beta_4<\cdots<\beta_{n-2}\leq\beta_{n-1}<\beta_n,\\
&&\alpha_1<\alpha_2\leq\alpha_3<\alpha_4\leq\cdots\leq\alpha_{n-2}<\alpha_{n-1}\leq\alpha_n,\nonumber
\end{eqnarray}
and similarly, when $n$ is even,
\begin{eqnarray}\label{n even}
&&\beta_1<\beta_2\leq\beta_3<\beta_4\leq\cdots<\beta_{n-2}\leq\beta_{n-1}<\beta_n,\\
&&\alpha_1\leq\alpha_2<\alpha_3\leq\alpha_4<\cdots\leq\alpha_{n-2}<\alpha_{n-1}\leq\alpha_n,\nonumber
\end{eqnarray}

We will only prove the inequality $(-1)^nu_i\geq -A_1 ~~\textup{on} ~~Z_{\varepsilon}$, the proof for $ (-1)^nu_i\leq A _2~~\textup{on} ~~X_{\varepsilon}$ is similar. Then we take $A:=\max\{A_1, A_2\}$.

We will firstly handle the case that there is only one $\beta_i$ with multiplicity $2$. Suppose when $f\rightarrow b$ from below, it happens that
$$\beta_1<\beta_2<\cdots<\beta_i=\beta_{i+1}<\beta_{i+2}<\cdots<\beta_n.$$
According to Lemma \ref{Y}, we will deal with $u_{pj}$ defined in (\ref{uij}) for each $p$ on $Z_{\varepsilon}$:

\noindent
(1) When $p=i$, observe that 
\begin{eqnarray*}
u_{i,i+1}%&=&\frac{1}{(\lambda_i-\lambda_{i+1})^2\prod\limits_{k=1; k\neq i,i+1}^n(\lambda_k-\lambda_{i+1})}\\
&=& (-1)^{i-1}\frac{1}{(\lambda_i-\lambda_{i+1})^2}\cdot\frac{1}{\prod\limits_{k<i}(\lambda_{i+1}-\lambda_k)}\cdot\frac{1}{\prod\limits_{k>i+1}(\lambda_k-\lambda_{i+1})}\\
&\rightarrow&(-1)^{i-1}(+\infty)\quad as~ f\rightarrow b.
\end{eqnarray*}
According to the explanations (\ref{n odd}) and (\ref{n even}), $(-1)^n=(-1)^i$, thus 
$$(-1)^{n+1}u_{i,i+1}\rightarrow +\infty\quad as~ f\rightarrow b.
$$
For $j\neq i, i+1$,
\begin{eqnarray*}
u_{ij}&=&\frac{1}{(\lambda_i-\lambda_j)^2\prod\limits_{k=1; k\neq i,j}^n(\lambda_k-\lambda_j)}\\
&\rightarrow&\frac{1}{(\beta_i-\beta_j)^2\prod\limits_{k=1; k\neq i,j}^n(\beta_k-\beta_j)}\quad as~ f\rightarrow b,
\end{eqnarray*}
which is a finite value. Therefore, as $f\rightarrow b$,
$$(-1)^nu_i=(-1)^{n+1}\frac{2(n-2)!}{n}\sum_{j=1; j\neq i}^n u_{ij}\rightarrow +\infty.$$

\noindent
(2) When $p=i+1$, observe that
\begin{eqnarray*}
u_{i+1,i}%&=&\frac{1}{(\lambda_{i+1}-\lambda_{i})^2\prod\limits_{k=1; k\neq i,i+1}^n(\lambda_k-\lambda_{i})}\\
&=& (-1)^{i-1}\frac{1}{(\lambda_{i+1}-\lambda_{i})^2}\cdot\frac{1}{\prod\limits_{k<i}(\lambda_{i}-\lambda_k)}\cdot\frac{1}{\prod\limits_{k>i+1}(\lambda_k-\lambda_{i})}\\
&\rightarrow&(-1)^{i-1}(+\infty)\quad as~ f\rightarrow b.
\end{eqnarray*}
According to the explanations (\ref{n odd}) and (\ref{n even}), $(-1)^n=(-1)^i$, thus 
$$(-1)^{n+1}u_{i+1,i}\rightarrow +\infty\quad as~ f\rightarrow b.$$
For $j\neq i, i+1$,
\begin{eqnarray*}
u_{i+1,j}&=&\frac{1}{(\lambda_{i+1}-\lambda_j)^2\prod\limits_{k=1; k\neq i+1,j}^n(\lambda_k-\lambda_j)}\\
&\rightarrow&\frac{1}{(\beta_{i+1}-\beta_j)^2\prod\limits_{k=1; k\neq i+1,j}^n(\beta_k-\beta_j)}\quad as~ f\rightarrow b
\end{eqnarray*}
which is a finite value. Therefore, as $f\rightarrow b$,
$$(-1)^nu_{i+1}=(-1)^{n+1}\frac{2(n-2)!}{n}\sum_{j=1; j\neq i+1}^n u_{i+1,j}\rightarrow +\infty.$$

\noindent
(3) When $p\neq i, i+1$, we see that 
\begin{eqnarray*}
u_{pi}%&=&\frac{1}{(\lambda_p-\lambda_{i})^2\prod\limits_{k=1; k\neq p,i}^n(\lambda_k-\lambda_{i})}\\
&=&\frac{1}{\lambda_{i+1}-\lambda_i}\cdot\frac{1}{(\lambda_p-\lambda_{i})^2\prod\limits_{k=1; k\neq p,i,i+1}^n(\lambda_k-\lambda_{i})},\\
u_{p,i+1}%&=&\frac{1}{(\lambda_p-\lambda_{i+1})^2\prod\limits_{k=1; k\neq p,i+1}^n(\lambda_k-\lambda_{i+1})}\\
&=&-\frac{1}{\lambda_{i+1}-\lambda_{i}}\cdot\frac{1}{(\lambda_p-\lambda_{i+1})^2\prod\limits_{k=1; k\neq p,i,i+1}^n(\lambda_k-\lambda_{i+1})},
\end{eqnarray*}
thus
$$u_{pi}+u_{p,i+1}=\frac{1}{\lambda_{i+1}-\lambda_i}\cdot\frac{(\lambda_p-\lambda_{i+1})^2\prod\limits_{k\neq p,i,i+1}(\lambda_k-\lambda_{i+1})-(\lambda_p-\lambda_{i})^2\prod\limits_{k\neq p,i,i+1}(\lambda_k-\lambda_{i})}{(\lambda_p-\lambda_{i})^2(\lambda_p-\lambda_{i+1})^2\prod\limits_{k\neq p,i,i+1}(\lambda_k-\lambda_{i})(\lambda_k-\lambda_{i+1})}.$$
Define 
\begin{eqnarray*}
H(x)&=&(\lambda_p-x)^2\prod\limits_{k\neq p,i,i+1}(\lambda_k-x)\\
&=&(-1)^{n-1}x^{n-1}+a_1x^{n-2}+a_2x^{n-3}+\cdots+a_{n-2}x+a_{n-1},
\end{eqnarray*}
where the coefficients $a_k=a_k(n, \lambda_1,\cdots,\lambda_{i-1},\lambda_{i+2},\cdots,\lambda_n)$, $k=1,\cdots,n-1$.

Therefore, the numerator of $u_{pi}+u_{p,i+1}$ is
\begin{eqnarray*}
&& H(\lambda_{i+1})-H(\lambda_i)\\
&=& (-1)^{n-1}(\lambda_{i+1}^{n-1}-\lambda_{i}^{n-1})+a_1(\lambda_{i+1}^{n-2}-\lambda_{i}^{n-2})+\cdots+a_{n-2}(\lambda_{i+1}-\lambda_{i})\\
&=&(\lambda_{i+1}-\lambda_{i})\cdot\Big((-1)^{n-1}(\lambda_{i+1}^{n-2}+\lambda_{i+1}^{n-3}\lambda_i+\cdots+\lambda_i^{n-2})+\cdots+a_{n-2}\Big)
\end{eqnarray*}
and thus when $f\rightarrow b$,
$$(-1)^{n+1}(u_{pi}+u_{p,i+1})\rightarrow\frac{(n-1)\beta_i^{n-2}+(-1)^{n+1}(n-2)\overline{a}_1\beta_i^{n-3}+\cdots+(-1)^{n+1}\overline{a}_{n-2}}{(\beta_p-\beta_i)^4\prod\limits_{k\neq p,i,i+1}(\beta_k-\beta_{i})^2},$$
where $\overline{a}_k=\overline{a}_k(n, \beta_1,\cdots,\beta_{i-1},\beta_{i+2},\cdots,\beta_n)$, $k=1,\cdots,n-2$. The limit is a finite value.

For $j\neq i, i+1$, we see that 
\begin{equation*}
u_{p,j}%\frac{1}{(\lambda_{p}-\lambda_j)^2\prod\limits_{k\neq p,j}(\lambda_k-\lambda_j)}
\rightarrow\frac{1}{(\beta_{p}-\beta_j)^2\prod\limits_{k\neq p,j}(\beta_k-\beta_j)}\quad as~ f\rightarrow b,
\end{equation*}
which is also a finite value.

Therefore, 
\begin{eqnarray*}
(-1)^nu_{p}%&=&(-1)^{n+1}\frac{2(n-2)!}{n}\sum_{j=1; j\neq p}^n u_{p,j}\\
&=&\frac{2(n-2)!}{n}\Big( (-1)^{n+1}(u_{pi}+u_{p,i+1})+(-1)^{n+1}\sum_{j\neq p,i,i+1}u_{pj}\Big)\\
&\rightarrow&
\frac{2(n-2)!}{n}\Big(\frac{(n-1)\beta_i^{n-2}+(-1)^{n+1}(n-2)\overline{a}_1\beta_i^{n-3}+\cdots+(-1)^{n+1}\overline{a}_{n-2}}{(\beta_p-\beta_i)^4\prod\limits_{k\neq p,i,i+1}(\beta_k-\beta_{i})^2}\\
&&\qquad\qquad+(-1)^{n+1}\sum_{j\neq p,i,i+1}\frac{1}{(\beta_{p}-\beta_j)^2\prod\limits_{k\neq p,j}(\beta_k-\beta_j)}\Big)\\
&:=&A_1(n,p),
\end{eqnarray*}
which is a finite value.

For sufficiently small $\varepsilon$, define
$$-A_1=-\max\limits_{p\neq i,i+1}\{|A_1(n,p)|\}-1,$$
then we arrive at
$$(-1)^{n}u_p\geq -A_1\qquad \textup{on}~Z_{\varepsilon}.$$

\vspace{2mm}

If there are more than one $\beta_i$ with multiplicity $2$, we only deal with the case that there are two $\beta_i$'s with multiplicity $2$, since the proof for the other cases are verbatim.

Suppose when $f\rightarrow b$ from below, it happens that
$$\beta_1<\cdots<\beta_i=\beta_{i+1}<\cdots<\beta_{s}=\beta_{s+1}<\cdots<\beta_n.$$

\noindent
(1) When $p=i$, we see that 
\begin{eqnarray*}
u_{i,i+1}%&=&\frac{1}{(\lambda_i-\lambda_{i+1})^2\prod\limits_{k=1; k\neq i,i+1}^n(\lambda_k-\lambda_{i+1})}\\
&=& (-1)^{i-1}\frac{1}{(\lambda_i-\lambda_{i+1})^2}\cdot\frac{1}{\prod\limits_{k<i}(\lambda_{i+1}-\lambda_k)}\cdot\frac{1}{\prod\limits_{k>i+1}(\lambda_k-\lambda_{i+1})}\\
&\rightarrow&(-1)^{i-1}(+\infty)\quad as~ f\rightarrow b.
\end{eqnarray*}
According to the explanations (\ref{n odd}) and (\ref{n even}), $(-1)^n=(-1)^i$, thus 
$$(-1)^{n+1}u_{i,i+1}\rightarrow +\infty\quad as~ f\rightarrow b\quad as~ f\rightarrow b.$$
As we discussed before, 
$$(-1)^{n+1}(u_{is}+u_{i,s+1})\rightarrow\frac{(n-1)\beta_s^{n-2}+(-1)^{n+1}(n-2)\overline{a}_1\beta_s^{n-3}+\cdots+(-1)^{n+1}\overline{a}_{n-2}}{(\beta_i-\beta_s)^4\prod\limits_{k\neq p,s,s+1}(\beta_k-\beta_{s})^2},$$
where $\overline{a}_k=\overline{a}_k(n, \beta_1,\cdots,\beta_{s-1},\beta_{s+2},\cdots,\beta_n)$, $k=1,\cdots,n-2$. The limit is a finite value.

For $j\neq i, i+1,s,s+1$, 
\begin{eqnarray*}
u_{ij}%&=&\frac{1}{(\lambda_i-\lambda_j)^2\prod\limits_{k=1; k\neq i,j}^n(\lambda_k-\lambda_j)}\\
&\rightarrow&\frac{1}{(\beta_i-\beta_j)^2\prod\limits_{k=1; k\neq i,j}^n(\beta_k-\beta_j)}\quad as~ f\rightarrow b
\end{eqnarray*}
which is a finite value. Therefore,
$$(-1)^nu_i=(-1)^{n+1}\frac{2(n-2)!}{n}\sum_{j=1; j\neq i}^n u_{ij}\rightarrow +\infty\quad as~ f\rightarrow b.$$

\noindent
(2) When $p=i+1$, $s$ or $s+1$, the discussion is similar and 
$$(-1)^nu_{i+1}\rightarrow +\infty,\quad (-1)^nu_{s}\rightarrow +\infty,\quad (-1)^nu_{s+1}\rightarrow +\infty\quad as~ f\rightarrow b.$$

\noindent
(3) When $p\neq i$, $i+1$, $s$ or $s+1$, as $f\rightarrow b$,
$$(-1)^{n+1}(u_{pi}+u_{p,i+1})\rightarrow\frac{(n-1)\beta_i^{n-2}+(-1)^{n+1}(n-2)\overline{a}_1\beta_i^{n-3}+\cdots+(-1)^{n+1}\overline{a}_{n-2}}{(\beta_p-\beta_i)^4\prod\limits_{k\neq p,i,i+1}(\beta_k-\beta_{i})^2},$$
$$(-1)^{n+1}(u_{ps}+u_{p,s+1})\rightarrow\frac{(n-1)\beta_s^{n-2}+(-1)^{n+1}(n-2)\overline{a}'_1\beta_s^{n-3}+\cdots+(-1)^{n+1}\overline{a}'_{n-2}}{(\beta_p-\beta_s)^4\prod\limits_{k\neq p,s,s+1}(\beta_k-\beta_{s})^2},$$
where $\overline{a}_k=\overline{a}_k(n, \beta_1,\cdots,\beta_{i-1},\beta_{i+2},\cdots,\beta_n)$, and $\overline{a}'_k=\overline{a}'_k(n, \beta_1,\cdots,\beta_{s-1},\beta_{s+2},\cdots,\beta_n)$, $k=1,\cdots,n-2$. The limits are both finite.

For $j\neq p, i, i+1,s,s+1$, we see that 
\begin{equation*}
u_{p,j}%=\frac{1}{(\lambda_{p}-\lambda_j)^2\prod\limits_{k\neq p,j}(\lambda_k-\lambda_j)}
\rightarrow\frac{1}{(\beta_{p}-\beta_j)^2\prod\limits_{k\neq p,j}(\beta_k-\beta_j)}\quad as~ f\rightarrow b,
\end{equation*}
which is also a finite value.

Therefore, 
\begin{eqnarray*}
(-1)^nu_{p}%&=&(-1)^{n+1}\frac{2(n-2)!}{n}\sum_{j=1; j\neq p}^n u_{p,j}\\
&=&\frac{2(n-2)!}{n}\Big( (-1)^{n+1}(u_{pi}+u_{p,i+1})+(-1)^{n+1}(u_{ps}+u_{p,s+1})\\
&&\qquad\qquad\quad+(-1)^{n+1}\sum_{j\neq p,i,i+1,s,s+1}u_{pj}\Big)\\
&\rightarrow&
\frac{2(n-2)!}{n}\Big(\frac{(n-1)\beta_i^{n-2}+(-1)^{n+1}(n-2)\overline{a}_1\beta_i^{n-3}+\cdots+(-1)^{n+1}\overline{a}_{n-2}}{(\beta_p-\beta_i)^4\prod\limits_{k\neq p,i,i+1}(\beta_k-\beta_{i})^2}\\
&&\qquad\qquad+\frac{(n-1)\beta_s^{n-2}+(-1)^{n+1}(n-2)\overline{a}'_1\beta_s^{n-3}+\cdots+(-1)^{n+1}\overline{a}'_{n-2}}{(\beta_p-\beta_s)^4\prod\limits_{k\neq p,s,s+1}(\beta_k-\beta_{s})^2}\\
&&\qquad\qquad+(-1)^{n+1}\sum_{j\neq p,i,i+1,s,s+1}\frac{1}{(\beta_{p}-\beta_j)^2\prod\limits_{k\neq p,j}(\beta_k-\beta_j)}\Big)\\
&:=&A_1(n,p),
\end{eqnarray*}
which is a finite value.

Again, for sufficiently small $\varepsilon$, define
$$-A_1=-\max\limits_{p\neq i,i+1,s,s+1}\{|A_1(n,p)|\}-1,$$
then we arrive at
$$(-1)^{n}u_p\geq -A_1\qquad \textup{on}~Z_{\varepsilon}.$$

\end{proof}


\begin{thebibliography}{123}

\bibitem[CCJ07]{CCJ07}
T. E. Cecil, Q. S. Chi and G. R. Jensen, \emph{Isoparametric
hypersurfaces with four principal curvatures}, Ann. Math.
\textbf{166} (2007), no. 1, 1--76.

\bibitem[CdK70]{CdK70}
S. S. Chern, M. do Carmo and S. Kobayashi, \emph{Minimal submanifolds of the sphere with second
fundamental form of constant length}, in: F. Browder (Ed.), Functional Analysis and Related
Fields, Springer-Verlag, Berlin, 1970.

\bibitem[Cha93]{Cha93}
S. P. Chang, \emph{On minimal hypersurfaces with constant scalar curvatures in $S^4$}, J. Differential Geom. \textbf{37}(1993), 523--534.


\bibitem[Cha93']{Cha93'}
S. P. Chang, \emph{A closed hypersurface with constant scalar curvature and constant mean curvature in $S^4$ is isoparametric},
Comm. Anal. Geom. \textbf{1}(1993), 71--100.


\bibitem[Che68]{Che68}
S. S. Chern, \emph{Minimal submanifolds in a Riemannian manifold}, Mimeographed Lecture Note, Univ.
of Kansas, 1968.

\bibitem[Chi11]{Chi11}
Q. S. Chi,\emph{Isoparametric hypersurfaces with four principal curvatures, II}, Nagoya
Math. J. \textbf{204} (2011), 1--18.

\bibitem[Chi13]{Chi13}
Q. S. Chi, \emph{Isoparametric hypersurfaces with four principal curvatures, III},  J. Differential Geom.  \textbf{94} (2013), 469--504.

\bibitem[Chi20]{Chi20}
Q. S. Chi, \emph{Isoparametric hypersurfaces with four principal curvatures, IV},
J. Differential Geom., \textbf{115}(2020), 225--301.

\bibitem[CR15]{CR15}
T. E. Cecil and P. J. Ryan, \emph{Geometry of hypersurfaces},
Springer Monographs in Mathematics, Springer, New York (2015).

\bibitem[CW93]{CW93}
Q. M. Cheng and Q. R. Wan, \emph{Hypersurfaces of space forms $M^4(c)$ with constant mean curvature},  Geometry and global analysis (Sendai, 1993), 437--442, Tohoku Univ., Sendai, 1993. 

\bibitem[CW]{CW}
Q. M. Cheng and G. X. Wei, \emph{Chern problems on minimal hypersurfaces}, preprint.


\bibitem[dB90]{dB90}
S. C. de Almeida, F. G. B. Brito, \emph{Closed 3-dimensional hypersurfaces with constant mean curvature and constant scalar curvature}, Duke Math. J. \textbf{61} (1990), 195--206.

\bibitem[DGW17]{DGW17}
Q. T. Deng, H. L. Gu and Q. Y Wei, \emph{Closed Willmore minimal hypersurfaces with constant scalar curvature in $S^5(1)$ are isoparametric}, Adv. Math. \textbf{314} (2017), 278--305.

\bibitem[DN85]{DN85}
J. Dorfmeister and E. Neher, \emph{Isoparametric hypersurfaces, case $g = 6, m =1$}, Comm. Algebra \textbf{13} (1985), 2299--2368.

\bibitem[DX11]{DX11}
Q. Ding and Y. L. Xin, \emph{On Chern's problem for rigidity of minimal hypersurfaces in the spheres}, Adv. Math., \textbf{227} (2011), 131--145.

\bibitem[GT12]{GT12}
J. Q. Ge, Z. Z. Tang, Chern conjecture and isoparametric hypersurfaces, in ``Differential Geometry--under the influence of S.S. Chern", edited by Y.B. Shen, Z.M. Shen, and S.T. Yau, Higher Education Press and International Press Beijing-Boston, 2012.


\bibitem[Imm08]{Imm08}
S. Immervoll, \emph{On the classification of isoparametric hypersurfaces with four distinct principal
curvatures in spheres}, Ann. Math., \textbf{168} (2008), 1011--1024.

\bibitem[Law69]{Law69}
H.B. Lawson Jr., \emph{Local rigidity theorems for minimal hypersurfaces}, Ann. Math. \textbf{89} (1969), 167--179.

\bibitem[LSS05]{LSS05}
T. Lusala, M. Scherfner, L.A.M. Sousa Jr., \emph{Closed minimal Willmore hypersurfaces of $S^5(1)$ with constant scalar curvature}, Asian J. Math. \textbf{9} (1) (2005), 65--78.

\bibitem[LXX17]{LXX17}
L. Lei, H. W. Xu and Z. Y. Xu, \emph{On Chern's conjecture for minimal hypersurfaces in spheres}, arXiv: 1712.01175.

\bibitem[Mil82]{Mil82}
J. Milnor, \emph{Hyperbolic geometry: the first 150 years}. Bull. Amer. Math. Soc. (N.S.) \textbf{6} (1982), no. 1, 9--24.

\bibitem[Miy13]{Miy13}
R. Miyaoka, \emph{Isoparametric hypersurfaces with (g,m) = (6,2)},
Ann. Math. \textbf{177} (2013), 53--110.

\bibitem[Miy16]{Miy16}
R. Miyaoka, \emph{Errata of ``isoparametric hypersurfaces with (g, m) = (6, 2) "},
Ann. Math. \textbf{183} (2016), 1057--1071.

\bibitem[Mun80]{Mun80}
H. F. M\"unzner, Isoparametrische Hyperfl\"achen in Sph\"aren, I, II, Math. Ann., \textbf{251}(1980), 57--71 and \textbf{256}(1981),  215--232.


\bibitem[PT83]{PT83}
C. K. Peng and C. L. Terng, \emph{Minimal hypersurfaces of spheres with constant scalar curvature}, Seminar on Minimal Submanifolds,
Ann. Math. Stud., Princeton Univ. Press, Princeton, NJ, 1983, 177--198.

\bibitem[PT83']{PT83'}
C. K. Peng and C. L. Terng, \emph{The scalar curvature of minimal hypersurfaces in spheres}, Math. Ann., \textbf{266}(1983), 105--113.

\bibitem[Sim68]{Sim68}
J. Simons, \emph{Minimal varieties in Riemannian manifolds}, Ann. Math. \textbf{88} (1968), 62--105.

\bibitem[SVW12]{SVW12}
M. Scherfner, L. Vrancken and S. Weiss, \emph{On closed minimal hypersurfaces with constant scalar curvature in $S^7$}. Geom. Ded. \textbf{161} (2012), 409--416.


\bibitem[SWY12]{SWY12}
M. Scherfner, S. Weiss and S.T. Yau, \emph{A review of the Chern conjecture for isoparametric hypersurfaces in spheres}, in: Advances in Geometric Analysis, in: Adv. Lect. Math. (ALM), vol.21, Int. Press, Somerville, MA, 2012, 175--187.

\bibitem[SY07]{SY07}
Y.J. Suh and H.Y. Yang, \emph{The scalar curvature of minimal hypersurfaces in a unit sphere}, Comm.
Contemp. Math. \textbf{9} (2007), 183--200.                                                                                                                                                         

\bibitem[TWY20]{TWY20}
Z. Z. Tang, D. Y. Wei and W. J. Yan, \emph{A sufficient condition for a hypersurface to be isoparametric}, Tohoku Math. J. \textbf{72} (2020), 493--505.


\bibitem[TY13]{TY13}
Z. Z. Tang and W. J. Yan, \emph{Isoparametric foliation and Yau conjecture on the first eigenvalue},  J. Differential Geom. \textbf{94}
(2013), 521--540.

\bibitem[TY15]{TY15}
Z. Z. Tang and W. J. Yan, \emph{Isoparametric foliation and a problem of Besse on generalizations of Einstein condition}, Adv. Math. \textbf{285} (2015), 1970--2000.

\bibitem[Ver86]{Ver86}
L. Verstraelen, \emph{Sectional curvature of minimal submanifolds}, In: Proceedings Workshop on Differential Geometry, Univ. Southampton, 1986, 48--62.

\bibitem[YC98]{YC98}
H. C. Yang and Q. M. Cheng, \emph{Chern's conjecture on minimal hypersurfaces}, Math. Z. \textbf{227} (1998), 377--390.

\bibitem[Yau82]{Yau82}
S. T. Yau, Problem section, In: Seminar on Differential Geometry, Ann. Math. Stud., 102, Princeton Univ. Press, Princeton, NJ, 1982, 669--706.

\bibitem[Yau14]{Yau14}
S. T. Yau, \emph{Selected Expository Works of Shing-Tung Yau with Commentary}, Vol 1, Advanced Lectures in Mathematics Series, Vol. 28, International Press of Boston, 2014.

\end{thebibliography}
\end{document}